\documentclass[12pt]{article}
\usepackage{amsmath,amsthm,amscd,amssymb}
\usepackage[mathscr]{eucal}
\usepackage{amssymb}
\usepackage{latexsym}
\theoremstyle{definition}
\newtheorem{Def}{Definition}[section]
\newtheorem{Thm}[Def]{Theorem}
\newtheorem{Prop}[Def]{Proposition}
\newtheorem{Rem}[Def]{Remark}
\newtheorem{Ex}[Def]{Example}
\newtheorem{Cor}[Def]{Corollary}

\newtheorem{Lem}[Def]{Lemma}

\numberwithin{equation}{section}

\begin{document}

\title{On mod $p$ singular modular forms}

\author{Siegfried B\"ocherer and Toshiyuki Kikuta}
\maketitle





\begin{abstract}
We show that an elliptic modular form with integral Fourier 
coefficients in a number field $K$, for which 
all but finitely many coefficients are divisible by a prime ideal $\frak{p}$ of $K$, 
is a constant modulo $\frak{p}$.
A similar property also holds for Siegel modular forms.
 Moreover, we define the notion of mod $\frak{p}$ 
singular modular forms and discuss some relations between their weights and the 
corresponding prime $p$. We discuss some examples of mod $\frak{p}$ singular 
modular forms arising from Eisenstein series and from theta series attached to lattices
with automorphisms. Finally, we apply our results to properties mod $\frak{p}$
of Klingen-Eisenstein series.     
\end{abstract}
 
\maketitle

\section{Introduction}
\label{intro}
Let $R$ be a commutative ring. An elliptic modular form $f=\sum_{n=0}^\infty a_f(n)q^n$ with all 
coefficients $a_f(n)$ in $R$ can be regard as an element of $R[\![q]\!]$. For example, we may consider
the normalized Eisenstein series $E_k$ of weight $k\ge 4$ for $SL_2(\mathbb{Z})$  
\begin{align*}
E_k=1-\frac{2k}{B_k}\sum _{n=0}^\infty \sum _{0<d|n}d^{k-1}q^n\in \mathbb{Q}[\![q]\!],   
\end{align*}  
where $B_k$ is the $k$-th Bernoulli number. In particular, if the weight is $p-1$ for a prime $p$ with $p\ge 3$, then $E_{p-1}\in \mathbb{Z}_{(p)}[\![q]\!]$ and  
\begin{align*}
E_{p-1}\equiv 1 \bmod{p},  
\end{align*}
in other words $\widetilde{E_{p-1}}=1\in \mathbb{F}_p^\times $. 
This is an important fact for the study of Galois representations 
attached to 
modular forms and $p$-adic modular forms 
(e.g. Swinnerton-Dyer \cite{Swn}, Serre \cite{Se}, etc.). 
It is  natural to ask  the question (Q) below: 

(Q)\ \textit{Does there exist a modular form $f$ such 
that $\widetilde{f}$ is  a non-constant 
 polynomial in $\mathbb{F}_p[q]$? }

It seems that the existence of such forms was generally doubted, 
however the question  was not settled in previous studies (as far as we know). 
One of our purposes is to answer  
this question for the case of Siegel modular forms of general 
degree including the 
case of half integral weight (see Theorem \ref{ThmM2}); we will also 
prove a similar
result for Jacobi forms. 

Moreover, we discuss the weights of ``mod ${p}$ 
singular Siegel modular forms'' (defined in Section \ref{Main}). 
We recall (see \cite{Frei}) 
that 
a Siegel modular form of degree $n$ called {\it singular} 
(in the usual sense) if all its rank $n$ Fourier coefficients vanish.
A main result on this topic is that modular forms of degree $n$ and weight $k$
are singular iff their weight is smaller than  
$\frac{n}{2}$. Our second purpose is to 
describe the possible weights  of  mod ${p}$ singular modular forms 
in a similar way. Actually, in  Theorem \ref{ThmM3} we prove a more 
subtle property: Suppose that a degree $n$ modular form of weight $k$ 
has coefficients in ${\mathbb Z}$ and there exists a rank $r$ Fourier
coefficient not divisible by $p$. Then the same property also holds for degree
$r+1$ unless $2k-r$ is divisible by $p-1$.  

We construct some examples of mod ${p}$ singular modular 
forms in two ways (see Section \ref{Ex}). 
Moreover as an application of Theorem \ref{ThmM3} on mod ${p}$ modular forms 
we describe 
some properties of Fourier coefficients of 
Klingen-Eisenstein series (see Theorem \ref{Kli}).  \\  
In this introduction, we used rational primes $p$ for congruences.
In the main text, we will deal with congruences modulo ${\mathfrak p}^m$, where
$\mathfrak p$ is a prime ideal in an arbitrary number field. This is possible
mainly due to the refinements of the theory of modular forms mod $p$
due to Rasmussen \cite{Ra}. Our focus will be on the case of odd primes $p$,
but we should mention, that most of our results are - after some modification -
also valid for $p=2$.

{\bf Acknowledgment:} This work was done when the first 
author held a guest professorship
at the Graduate School of Mathematical Sciences at the University of 
Tokyo. He wishes to
thank Professor T. Oda for arranging this stay and 
supporting this collaboration. We also thank Professor S. Nagaoka
for helpful discussions.

\section{Preliminaries}
\subsection{Siegel modular forms}
\label{Sie}
Let $N$ be a natural number. In this paper, we deal with three types congruence subgroups of Siegel modular group $\Gamma _n=Sp_n(\mathbb{Z})$ as follows: 
\begin{align*}
\Gamma ^{(n)}(N)&:=\left\{ \begin{pmatrix}A & B \\ C & D \end{pmatrix}\in \Gamma _n \: \Big{|} \: B\equiv C \equiv 0_n \bmod{N},\ A\equiv D\equiv 1_n \bmod{N} \right\},\\
\Gamma _1^{(n)}(N)&:=\left\{ \begin{pmatrix}A & B \\ C & D \end{pmatrix}
\in \Gamma _n \: \Big| \: C\equiv 0_n \bmod{N},\ \det (A) \equiv \det (D) \equiv 1 \bmod{N} \right\},\\
\Gamma _0^{(n)}(N)&:=\left\{ \begin{pmatrix}A & B \\ C & D \end{pmatrix}\in \Gamma _n \: \Big| \: C\equiv 0_n \bmod{N} \right\}.  
\end{align*} 
Let $\Gamma $ be the one of above modular groups of degree $n$ with level $N$. For a natural number $k$ and a Dirichlet character $\chi : (\mathbb{Z}/N\mathbb{Z})^\times \rightarrow \mathbb{C}^\times $, the space $M_k(\Gamma ,\chi )$ of Siegel modular forms of weight $k$ with character $\chi$ consists of all of holomorphic functions $f:\mathbb{H}_n\rightarrow \mathbb{C}$ satisfying  
\begin{equation}
f(MZ)=\chi (\det D)\det (CZ+D)^kf(Z),\quad for\ M=\begin{pmatrix}A & B \\ C & D \end{pmatrix}\in \Gamma , 
\label{modular}\end{equation}
where $\mathbb{H}_n$ is the Siegel upper-half space of degree $n$. If $n=1$, the usual condition in the cusps should be added.

If $k=l/2$ is half-integral, then we assume that the level $N$ of $\Gamma $ satisfies $4|N$. 
We define $M_k(\Gamma ,\chi )$ as the space of all of holomorphic functions $f:\mathbb{H}_n\rightarrow \mathbb{C}$ such that  
\begin{align*}
f(MZ)=\chi (\det D) j(M,Z)^l f(Z),\quad for\ M=\begin{pmatrix}A & B \\ C & D \end{pmatrix}\in \Gamma , 
\end{align*} 
where $j(M,Z)$ defined by the transformation law of the theta series 
\[\theta ^{(n)}(Z):=\sum _{X\in \mathbb{Z}^n}e^{2\pi i {}^tX Z X}, 
\]
namely \[j(M,Z):=\frac{\theta ^{(n)}(MZ)}{\theta ^{(n)}(Z)},\quad for \ M\in \Gamma _0^{(n)}(4). \]
For more details  on Siegel modular forms of half-integral weight, see \cite{And-Zu}. 

In both cases, when $\chi $ is a trivial character, we write simply as $M_k(\Gamma )$ for $M_k(\Gamma ,\chi )$. 
  
Any $f \in M_k(\Gamma, \chi )$ has a Fourier expansion of the form
\[
f(Z)=\sum_{0\leq T\in\Lambda_n}a_f(T)e^{2\pi i \frac{1}{N}{\rm tr}(TZ)},
\quad Z\in\mathbb{H}_n,
\]
where
\[
\Lambda_n
:=\{ T=(t_{ij})\in Sym_n(\mathbb{Q})\;|\; t_{ii},\;2t_{ij}\in\mathbb{Z}\; \}
\]
(the lattice in $Sym_n(\mathbb{R})$ of half-integral, symmetric matrices). In particular, if $\Gamma $ is among  $\Gamma _n$, $\Gamma _0^{(n)}(N)$ and $\Gamma _1^{(n)}(N)$, then the 
Fourier expansion of $f$ is given by the form 
\[
f(Z)=\sum_{0\leq T\in\Lambda_n}a_f(T)e^{2\pi i {\rm tr}(TZ)}.
\]

We denote by $\Lambda _n^+$ the set of all positive definite elements 
of $\Lambda _n$. We call two elements $S,T$ of $\Lambda_n$ equivalent, if
there exits $U\in GL(n,{\mathbb Z})$ such that $S[U]=T$, 
where $S[U]:= {}^tUSU$.\\ 
For a subring $R$ of $\mathbb{C}$, let $M_{k}(\Gamma ,\chi )_{R}
\subset M_{k}(\Gamma, \chi )$ 
denote the space of all modular forms whose Fourier coefficients are in $R$.

\subsection{Jacobi forms and their theta expansion}
\label{Jacobi}

In this subsection we briefly recall some properties of Jacobi forms and their
theta expansion. For details we refer
to \cite{Zie}, generalizing the one-dimensional case from \cite{Ei-Za}.

For a congruence subgroup $\Gamma $ of $Sp_n({\mathbb Z})$ and an element
$S\in \Lambda^+_r$ we define degree $n$ Jacobi forms of index $S$ and 
weight $k$ for $\Gamma$ as holomorphic functions
$\varphi:{\mathbb H}_n\times {\mathbb C}^{(n,r)}\rightarrow \mathbb{C}$ such that for 
$Z=\left( \begin{smallmatrix} \tau &{\mathfrak z}\\ {}^t {\mathfrak z} & 
\tau'\end{smallmatrix} \right)$, $\tau \in {\mathbb H}_n,\tau' \in {\mathbb H}_r$
$$F(Z):= \varphi(\tau, {\mathfrak z}) e^{2\pi i {\rm tr} (S\cdot \tau')}$$
satisfies the usual transformation law (\ref{modular}) 
of modular forms of degree $n+r$ for the group $\Gamma_{n,r}^J$ defined by
\begin{equation}\left\{ \begin{pmatrix} a_1 & 0 &{}& b_1 & b_2\\
a_3 & 1_r &{}&  b_3 & b_4\\[0.2cm]
c_1 & 0 &{}& d_1 & d_2\\
0 & 0 & {}& 0 & 1_r\end{pmatrix}\in Sp_{n+r}({\mathbb Z})\, \Big{|} 
\begin{pmatrix} a_1 & b_1\\ c_1 & d_1\end{pmatrix}
\in 
\Gamma \right\}\label{jacobiproblem}
\end{equation} 
with the usual additional condition in the cusps when $n=1$. 
We denote the space of such functions by $M_{k,S}(\Gamma^J_{n,r})$.

The Fourier expansion of such Jacobi forms can be written as
$$\varphi(\tau,{\mathfrak z})=\sum_{(T,R)\in \Lambda_n\times {\mathbb Z}^{(n,r)}}
c(T,R)e^{2\pi i {\rm tr}(T\tau +R {}^t {\mathfrak z})}.$$
Then the Fourier coefficients satisfy
$$c(T,R)=\det(a_1)^k c(L,M),$$ {if}
\begin{equation}
\left(\begin{array}{cc} T & \frac{R}{2}\\
\frac{{}^t R}{2} & S\end{array}\right) \left[ \left(\begin{array}{cc} a_1 & 0\\
a_3 & 1_r\end{array}\right) \right] =
\left(\begin{array}{cc} L & \frac{M}{2}\\
\frac{{}^tM}{2} & S\end{array}\right)
\label{Jequiv}
\end{equation}
with
$\left( \begin{smallmatrix} 
a_1 & 0 &{}& 0 & 0\\
a_3 & 1_r &{} & 0 & 0\\
0 & 0 & {} & {}^ta_1^{-1} & d_2 \\
0 & 0 &{} & 0 & 1_r \end{smallmatrix} \right)\in \Gamma^J_{n,r}$.
We then call $(T,R)$ and $(L,M)$ J-equivalent. \\
Jacobi forms have a ``theta expansion''
$$\varphi(\tau,{\mathfrak z})=\sum_{\mu} h_{\mu}(\tau)\Theta_S[\mu]
(\tau,{\mathfrak z}),$$
where $\mu$ runs over ${\mathbb Z}^{(n,r)}\cdot (2S)\backslash 
{\mathbb Z}^{(n,r)}$
and 
$$\Theta_S[\mu](\tau,{\mathfrak z})=\sum_{\lambda \in {\mathbb Z}^{(n,r)}} 
e^{2\pi i {\rm tr}( S[{}^t \lambda +(2S)^{-1}{}^t \mu] \tau
+ 2 (\lambda+\mu (2S)^{-1}) {}^t{\mathfrak z}}. $$

The $h_{\mu}$ are then modular forms of weight $k-\frac{r}{2}$ 
(possibly half-integral) for the congruence subgroup $\Gamma\cap \Gamma^{(n)}(4L)$, where $L$ is the level
of the index $S$. Moreover, the Fourier expansion of $h_{\mu}$ is explicitly described by

$$h_{\mu}(\tau)= e^{-2\pi i {\rm tr}(\frac{1}{4}S^{-1}[{}^t \mu ]\tau)}
\sum_{T\in \Lambda_n} c(T,\mu) e^{2\pi i {\rm tr}(T \cdot \tau)}$$

Jacobi forms arise naturally as Fourier-Jacobi 
coefficients of Siegel modular forms and play an important role for us, see 
subsection \ref{subsec4.2}.

For a subring $R$ of $\mathbb{C}$, let $M_{k,S}(\Gamma^J_{n,r})_R
\subset M_{k,S}(\Gamma^J_{n,r})$ denote the space of all Jacobi forms whose Fourier coefficients are in $R$.

\section{Main results}
\label{Main} \label{sing}
Let $K$ be a number field with ring of integers ${\mathcal O}=
{\mathcal O}_K$. For a prime ideal $\frak{p}$ in 
${\mathcal O}$ with ${\mathfrak p} | p$, 
let ${\mathcal O}_{\frak{p}}$ be the 
localization of ${\mathcal O}$ at $\frak{p}$.
To formulate higher congruences for weights 
(following \cite{Ra}), we need the following notation:  
We denote by $e$ the ramification index $e=e({\mathfrak p}/p)$; for the 
Galois closure $L$ of $K$ let ${\mathfrak P}$ be a prime ideal of 
${\mathcal O_L}$ dividing ${\mathfrak p}$; we put $r:= v_p(\tilde{e})$,
where $\tilde{e}=e({\mathfrak P}/p)$. 
For a positive integer $m$ we define a nonnegative integer $\beta(m)$ 
by
$$\beta(m):=  \max \left \{ \Big \lceil \frac{m}{e} \Big \rceil -r-1, 0\right \}$$
We note that for the case $K=\mathbb Q$ we have $\beta(m)=m-1$.
     
First, we shall state our results on the question (Q) in the Introduction: 
The first our main result is 
\begin{Thm} 
\label{ThmM2}
Let $\frak{p}$ be an any prime ideal, 
$0< k \in \mathbb{Z}$ or $\frac{1}{2}+\mathbb{Z}$ and 
$f \in M_k(\Gamma ^{(n)}(N))_{{\mathcal O}_{\frak{p}}}$, 
(where $4|N$ unless $k\in \mathbb{Z}$). If there are at 
most only finitely many inequivalent $S_1, \dots , 
S_h\in \Lambda _n$ such that $a_f(S_i)\not \equiv 0$ mod $\frak{p}^m$, then $f\equiv c$ mod 
$\frak{p}^m$ for some $c\in {\mathcal O}_\frak{p}$. 
Moreover, if $k$ is integral, $c$ a unit of ${\mathcal O}_{\frak{p}}$ and $N$ coprime to $p\not=2$, 
then $(p-1)p^{\beta(m)}$ divides
the weight $k$. 
\end{Thm} 

\begin{Rem}
(1) It is quite natural to consider equivalence classes here because of the
invariance property 
$a_f(T)=a_f(T[U])$, valid  for all $T\in \Lambda _n$ and for infinitely many 
$U\in GL_n(\mathbb{Z})$.  To be more precise, for the group $\Gamma ^{(n)}(N)$ we 
should use a modified equivalence relation (allowing only such 
$U$, which are congruent to the unit matrix modulo $N$), but for the finiteness
condition in Theorem \ref{ThmM2} this does not really matter. 
\\
(2) The case of $n=1$ in Theorem \ref{ThmM2} 
answers the question (Q) in the Introduction. 
\end{Rem} 

We may ask a similar question for Jacobi forms:

\begin{Thm}
\label{ThmJ}
For $\Gamma:=\Gamma^{(n)}_1(N)$ and $S\in \Lambda_n^+$ let 
$$\varphi(\tau,{\mathfrak z})=
\sum_{R,T}c(T,R) e^{2\pi i tr(R\tau+T{\mathfrak z}^t)}=
\sum_{\mu}h_{\mu}(\tau)\Theta_S[\mu](\tau,{\mathfrak z})$$
be a Jacobi form in $M_{k,S}(\Gamma_{n,r}^J)_{{\mathcal O}_\frak{p}}$ such that
$c(T,R)\equiv 0$ mod ${\mathfrak p}^m$ for all but 
finitely many J-equivalence classes $(T,R)$. Then
$$\varphi(\tau,{\mathfrak z})\equiv \sum_{\mu} a(\mu)\Theta_S[\mu](\tau,{\mathfrak z}) \bmod {\mathfrak p}^m$$ 
where the $a(\mu)$ are elements of ${\mathcal O}_{\mathfrak p}$. 
The sum extends only over those $\mu$
with $\frac{1}{4}S^{-1}[\mu]$ half-integral; 
the $a(\mu)$ are elements of ${\mathcal O}_{\mathfrak p}$, namely
$a(\mu)=c(\frac{1}{4}S^{-1}[\mu],\mu)$. 
Moreover, if $N$ and the level of S are both coprime to $p$ with $p\not=2$, 
and $2k-r$
is not divisible by $(p-1)p^{\beta(m)}$ then all Fourier 
coefficients $c(T,R)$ are all divisible by $\frak{p}^m$.
\end{Thm}

\begin{Rem} A result of similar type (for $m=n=1$) was given by
Choi, Choie and Richter \cite[Proposition 3]{C-C-R}; they imposed a congruence condition on
all but finitely many Fourier coefficients (not on J-equivalence classes);
under this stronger condition they show that the Jacobi form is congruent 
zero modulo ${\mathfrak p}$.
\end{Rem}

Next, we shall state our results on ``mod $\frak{p}$ singular modular forms'': 

Let $v_\frak{p}$ be the normalized additive valuation with respect to $\frak{p}$. We define two values $v_\frak{p}(f)$ and $v_\frak{p}^{(n')}(f)$ for $f\in M_k(\Gamma ,\chi )_K$ by 
\begin{align*}
&v_\frak{p}(f):=\min \{v_\frak{p}(a_f(T))\;|\;T \in \Lambda _{n} \}, \\
&v^{(n')}_\frak{p}(f):=\min \{v_\frak{p}(a_f(T))\;|\;T \in \Lambda _{n},\ {\rm rank}(T)=n' \}\quad (0\le n'\le n). 
\end{align*}
Then we have our second main theorem: 
\begin{Thm}
\label{ThmM3}
Let $p$ be an odd prime and $0< k \in \mathbb{Z}$ or $\frac{1}{2}+\mathbb{Z}$. Assume that $\psi ^2$ has a conductor coprime to $p$. Let $f\in M_k(\Gamma _0^{(n)}(N),\psi )_{K}$ ($4|N$ unless $k\in \mathbb{Z}$) and $\frak{p}$ be a prime ideal with $\frak{p}|p$ in $K$. Then 
$v _\frak{p}^{(n'+1)}(f)=v_\frak{p}^{(n')}(f)+m$ ($n'<n$) with $m\geq 1$ is 
possible only if $2k-n'\equiv 0$ mod $(p-1)p^{\beta(m)}$.  
\end{Thm}
Using this theorem, we get properties of the weights of 
``mod $\frak{p}$ singular modular forms'' defined as
\begin{Def}
If $f\in M_k(\Gamma , \chi )_{{\mathcal O}_{\frak{p}}}$ satisfies that $a_f(T)\equiv 0$ mod $\frak{p}$ for all $T\in \Lambda _n^+$, 
then we say that $f$ is $mod$ $\frak{p}$ $singular$ $modular$ $form$. We define the $\frak{p}$-rank of $f$ by the maximum of ranks $r$ of $T\in \Lambda _n$ such that there is a Fourier coefficient $a_f(T)$ not congruent to $0$ modulo $\frak{p}$. 
\end{Def} 
From Theorem \ref{ThmM3} we get 
\begin{Cor}
\label{Cor3}
Let $p$ be an odd prime and $0< k \in \mathbb{Z}$ or $\frac{1}{2}+\mathbb{Z}$. Assume that $f\in M_k(\Gamma _0^{(n)}(N),\psi )_{{\mathcal O}_{\frak{p}}}$ ($4|N$ unless $k\in \mathbb{Z}$) is a mod $\frak{p}$ singular modular form with $\frak{p}$-rank $r$ ($r<n$) and $\psi ^2$ has a conductor coprime to $p$, where $\frak{p}$ is a prime ideal with $\frak{p}|p$ in $K$. Then
\[2k-r\equiv 0 \bmod{p-1}.\]
In particular, only $r$ even (odd) may occur if $k\in \mathbb{Z}$ ($k\not \in \mathbb{Z}$).   
\end{Cor}
As an easy consequence of Theorem \ref{ThmM2} and Corollary \ref{Cor3}, we have
\begin{Cor}
Let $p$ be an odd prime and $0< k \in \mathbb{Z}$. Assume that $N$ is coprime to $p$ and $f\in M_k(\Gamma _1^{(2)}(N))_{\mathbb{Z}_{(p)}}$. 
If there are at most only finitely many inequivalent $S_1, \dots , S_h\in \Lambda _2^+$ 
such that $a_f(S_i)\not \equiv 0$ mod $p$, then $f\equiv c$ mod $p$ for some $c\in \mathbb{Z}_{(p)}$.  
\end{Cor}

\section{Proof of Theorem \ref{ThmM2}}
\subsection{The case $\boldsymbol{n=1}$}
First we shall prove 
\begin{Prop} 
\label{Prop}
Let $f \in M_k(\Gamma _1^{(1)}(N))_{\mathcal{O}_\frak{p}}$ with $0<k\in \mathbb{Z}$ or $\frac{1}{2}+\mathbb{Z}$. 
If $a_f(n)\equiv 0$ mod $\frak{p}^m$ for all but finitely many $n$, 
then $f\equiv c$ mod $\frak{p}^m$ for some $c\in \mathcal{O}_\frak{p}$. 
\end{Prop}
\begin{proof}[Proof of Proposition \ref{Prop}]
We take a number $s_0>0$ satisfying the following property: 

\textit{If $f \in M_k(\Gamma _1^{(1)}(N))_{\mathcal{O}_\frak{p}}$ satisfies that 
$a_f(n)\equiv 0$ mod $\frak{p}^m$ for all $n$ with $0\le n\le s_0$, then $f\equiv 0$ mod $\frak{p}^m$.}

For a modular form of integral weight, we may chose $s_0$ as the Sturm bound $s_0:=[k/12[SL_2(\mathbb{Z}):\Gamma _1^{(1)}(N)]]\in \mathbb{Z}_{\ge 0}$ (cf. Sturm \cite{St}). 
Note that Sturm stated the existence of such a bound only for $m=1$. However Rasmussen \cite{Ra} studied the case of general $m$. 

For a modular form $f$ of half-integral weight we can obtain an analogue of the Sturm type bound as above
by considering $f\cdot \theta$, where $\theta$ is the standard theta series of weight $\frac{1}{2}$. 
(It does not suffice to consider $f^2$ to get such a bound.)

We shall consider the case $k\in \frac{1}{2}+\mathbb{Z}$. 
By the assumption, we can take $d:=\min \{d\in \mathbb{Z}_{\ge 0} \;|\; 
\forall n>d,\ a_f(n)\equiv 0 \bmod{\frak{p}^m} \}$. 
If $m=1$, then $d$ is the usual degree of the polynomial $\widetilde{f}$. 
For a prime $l$ coprime to $N$, we denote by $f\mapsto f|T_{l^2}$ the action of the 
Hecke operator $T_{l^2}$ on $f\in M_k(\Gamma_1^{(1)}(N))$ as explained in \cite[IV,\S3]{Kob}. Then we can find two 
primes $l_1$, $l_2$ ($l_1>l_2>d$) such that $l_1 \equiv l_2\equiv 1$ mod $pN$ and 
$f|T_{l_1^2}\equiv f|T_{l_2^2}$ mod $\frak{p}^m$. 
The existence such the primes follows from the Sturm bound. In fact, for the Sturm bound $s_0$, 
\begin{align*}
&\# \left \{\left(\widetilde{a_{f|T_{l^2}}(0)}, \widetilde{a_{f|T_{l^2}}(1)}, \cdots, \widetilde{a_{f|T_{l^2}}(s_0)}\right) \in ({\mathcal O}/\frak{p}^m {\mathcal O})^{s_0+1} \: \Big{|} \: l>d:prime,\ l\equiv 1 \bmod{pN} \right \}\\
&\le p^{(s_0+1)g},
\end{align*}
where $p^g:=\# ({\mathcal O}/\frak{p}^m {\mathcal O})$.  
Hence, in the infinite set $\{l:prime \:|\: l \equiv 1 \bmod{pN}\}$, one can take two primes $l_{1}$ and $l_{2}$ such that the first $s_0+1$ Fourier coefficients of $f|T_{l_i^2}$ are the same modulo $\frak{p}^m$ according to the 
pigeonhole principle. For these primes, we have $f|T_{l_{1}^2}\equiv f|T_{l_{2}^2}$ mod $\frak{p}^m$. 

Now recall that 
\[M_k(\Gamma _1^{(1)}(N))=
\bigoplus _\chi M_k(\Gamma _0^{(1)}(N),\chi ),\]
where $\chi$ runs over all Dirichlet characters modulo $N$. (This decomposition holds for both cases $k\in \frac{1}{2}+\mathbb{Z}$ and $k\in \mathbb{Z}$.) Furthermore, there is a simple formula for the action of $T_{l^2}$ on elements of $M_k(\Gamma_0^{(1)}(N),\chi)$,
see \cite[IV,prop.13]{Kob}. The condition $\chi (l_i)=1$ ($l_i\equiv 1$ mod $N$) for each character $\chi$ modulo $N$, 
allows us to apply this formula to $f\in M_k(\Gamma^{(1)}_1(N))$:
\begin{align*}
f|T_{l_i^2}
&=\sum _{n=0}^{\infty } a_f(l_i^2 n)q^n+l_i^{k-\frac{3}{2}}\sum _{n=0}^{\infty }\left( \frac{(-1)^{k-\frac{1}{2}}n}{l_i}\right)a_f(n)q^{n}+l_i^{2k-2}\sum _{n=0}^{\infty }a_f(n)q^{l_i^2n}.
\end{align*}
Note that $a_f(l_i^2 n)\equiv 0$ mod $\frak{p}^m$ for all $n\ge 1$ because of $l_i>d$. Since the $l_2^2d$ th coefficient in $f|T_{l_1^2}$ does not appear and that of $f|T_{l_2^2}$ is $l_2^{2k-2}a_f(d)$ for $d\neq 0$. The condition $l_2\equiv 1$ mod $pN$ implies $l_2\in {\mathcal O}_{\frak{p}}^{\times }$. Hence, $a_f(d)\equiv 0$ mod $\frak{p}^m$ unless $d=0$. However, if $d\neq 0$, this is a contradiction for the choice (minimality) of $d$. 
Therefore it must be $d=0$ and hence $f\equiv c$ mod $\frak{p}^m$ for some $c\in {\mathcal O}_{\frak{p}}$.  

For integral weights, one can proceed as above, using the Hecke operators $T_{l_i^2}$.
It is however somewhat easier to work with the usual Hecke 
operators $T_{l_i}$ in this case, using 
\[f|T_{l_
i}=\sum _{n=0}^{\infty } a_f(l_i n)q^n+l^{k-1}
\sum _{n=0}^{\infty }a_f(n)q^{l_i n}.\]
This formula is well-known, see e.g. \cite[p.171]{DiSh} for integral weights, $\Gamma_1^{(1)}(N)$ and $l_i\equiv 1\bmod N$. 
One may now argue in a way quite similar to the half-integral weight case.
\end{proof}

\begin{proof}[Proof of the case $n=1$ of Theorem \ref{ThmM2}]
To cover the case of principal congruence subgroups, we observe that
$\Gamma ^{(1)}(N)$ is conjugate by 
$\left( \begin{smallmatrix} N & 0\\
0 & 1\end{smallmatrix}\right)$ to 
\[\left \{ \begin{pmatrix} a & b \\ c & d \end{pmatrix} \;\Big{|}\; 
a\equiv d \equiv 1 \bmod{N},\ c\equiv 0\bmod{N^2} \right \}
\supset \Gamma _1^{(1)}(N^2).\] 
Therefore, the assertion of the theorem \ref{ThmM2} follows for 
$f\in M_{k}(\Gamma ^{(1)}(N))$ with $0<k\in \mathbb{Z}$ or $\frac{1}{2}+\mathbb{Z}$ from the propositions above, applied to $f(N\tau)$, which is a modular form for
$\Gamma^{(1)}_1(N^2)$.
\end{proof}

\subsection{The case of general ${\boldsymbol n}$}
\label{subsec4.2}
We start with proving that
\begin{Prop}
\label{Prop4}
Let $k$ be a positive integer and 
$f\in M_{k}(\Gamma _1^{(n)}(N))_{{\mathcal O}_{\frak{p}}}$. 
If there are at most only finitely many inequivalent 
$S_1, \dots , S_h\in \Lambda _n^+$ such that 
$a_f(S_i)\not \equiv 0$ mod $\frak{p}^m$, 
then $a_f(T)\equiv 0$ mod $\frak{p}^m$ for all 
$T\in \Lambda _n^+$. Here $\Lambda _n^+$ is 
the set of all positive definite elements of $\Lambda _n$. 
\end{Prop}
\begin{proof}[Proof of Proposition \ref{Prop4}] 
We recall the well-known fact \cite{Frei, Kli, Zie} that 
$f$ has a Fourier-Jacobi expansion of the form
\[f(Z)=\sum _{M\in \Lambda _{n-1}}\varphi _M(\tau , \frak{z})e^{2 \pi i {\rm tr}(M\cdot \tau ')}.\] 
Here, we decomposed $\mathbb{H}_n\ni Z=\left( \begin{smallmatrix} \tau & {}^t \frak{z} \\ \frak{z} & \tau ' \end{smallmatrix} \right) $ for $\tau \in \mathbb{H}_1$ and $\tau '\in \mathbb{H}_{n-1}$.  
Now we pick up the $M_0$-th Fourier-Jacobi coefficient for any $M_0\in \Lambda _{n-1}^+$ and consider its theta expansion; 
\[ \varphi _{M_0}(\tau ,\frak{z})=\sum _{ \mu } h_\mu (\tau )\Theta _{M_0}[\mu ](\tau,\frak{z}) \]
with 
\begin{align*}
h_\mu (\tau )=\sum _{l=0}^\infty a_f\left( \begin{smallmatrix} l & \frac{\mu}{2} \\ \frac{{}^t \mu}{2} & M_0 \end{smallmatrix}\right)e^{2\pi i (l-\frac{1}{4}M_0^{-1}[\mu ])\tau }. 
\end{align*}
Note that $h_{\mu }$ is a modular form of weight $k-\frac{n-1}{2}$ 
for $\Gamma ^{(1)}(4NL)$, where $L$ is the level of $M_0$. 

Now $a_f(T)\not \equiv 0$ mod $\frak{p}^m$ only for finitely many inequivalent $T\in \Lambda _n^+$ (and hence only finitely many determinants) means for the Fourier coefficients of $h_{\mu}$ that at most only finitely many of them are not congruent to zero modulo $\frak{p}^m$. 

This follows from 
\[\det (T)=\left( l-\frac{1}{4}M_0^{-1}[\mu ] \right)\det (M_0) \]
if $T=\left(\begin{smallmatrix} l & \frac{\mu}{2} \\ \frac{{}^t\mu}{2} & M_0 \end{smallmatrix}\right)$. 
By the result on degree $1$ (Proposition \ref{Prop}), $h_{\mu }\equiv c$ mod $\frak{p}^m$ for some $c\in {\mathcal O}_{\frak{p}}$. Therefore only the Fourier coefficients with $\det (T)=0$ are 
possibly not congruent to zero modulo $\frak{p}^m$.  
\end{proof}
Using Proposition \ref{Prop4}, we can prove Theorem \ref{ThmM2} for 
$\Gamma _1^{(n)}(N)$:
\begin{proof}[Proof of Theorem \ref{ThmM2}]
We consider $f\in M_k(\Gamma _1^{(n)}(N))$ with $k\in \mathbb{Z}$. 

We may show that, if there exists $S_i$ with ${\rm rank}(S_i)=r \le n$ 
such that $a_f(S_i)\not \equiv 0$ mod $\frak{p}^m$, then there 
exist infinitely many inequivalent classes of quadratic forms 
in $\Lambda _{n}$ with rank $r$ such that corresponding 
Fourier coefficients are not congruent to zero modulo 
$\frak{p}^m$. For the Siegel $\Phi $-operator, 
we consider $\Phi ^{n-r}(f)\in 
M_k(\Gamma _1^{(r)}(N))_{{\mathcal O}_{\frak{p}}}$. 
By the assumption on the rank of $S_i$, there exists 
$U\in SL_n(\mathbb{Z})$ such that 
\begin{equation}S_i[U]=\left(\begin{smallmatrix}0 & 0 \\ 0 & T_i
\end{smallmatrix}\right) \label{problem}
\end{equation} for some $T_i\in \Lambda _r^+$. 
Then $a_{\Phi ^{n-r}(f)}(T_i) =a_f(S_i)\not \equiv 0$ 
mod $\frak{p}^m$. 
Applying Proposition \ref{Prop4} to $\Phi ^{n-r}(f)$, 
there are infinitely many inequivalent 
$T\in \Lambda _r^+$ such that $a_{\Phi ^{n-r}(f)}(T)
\not \equiv 0$ mod $\frak{p}^m$. For each $T$, 
we have $a_f\left( \begin{smallmatrix} 0 & 0 \\ 0 & T 
\end{smallmatrix}\right)\not \equiv 0$ mod $\frak{p}^m$. 
Hence the claim follows.

To get the general case $\Gamma=\Gamma^{(n)}(N)$, we need some minor technical
modifications in the proof above. First of all, the Fourier-Jacobi expansion
is somewhat more complicated, because (in the notation of (\ref{jacobiproblem}))
the elements $a_3$ will all be divisible by $N$.
Secondly, the matrix $U$ used in (\ref{problem}) does not necessarily
come from an element $\left(\begin{smallmatrix} U & 0\\
0 & {}^tU^{-1}\end{smallmatrix}\right)\in \Gamma^{(n)}(N)$. We have to
use a modified $\Phi$-operator (in another cusp). Up to these modifications,
the proof above works for the principal congruence subgroup as well.\\
As for half-integral weights, an inspection of the proof above shows that 
it works in the same way for half-integral weights. Note that
one should not go to a power of $f$ (the
finiteness condition gets lost for $n>1$ (!)). 
\end{proof}

\subsection{Proof of Theorem \ref{ThmJ} }
The proof is quite similar to the one for (the degree $n$ case of) Theorem \ref{ThmM2}:

For a given $\mu$, the Fourier 
coefficients of $h_{\mu}$ are indexed by 
$T-\frac{1}{4}S^{-1}[\mu]$,
assigning to it the value of $c(T,\mu)$; this is best described by 
the ``Jacobi-coordinates'' (with $R=\mu$)
$$ (T,R)\leftrightarrow 
\begin{pmatrix} T & \frac{R}{2}\\\frac{{}^tR}{2} & S\end{pmatrix}=
\begin{pmatrix} T-\frac{1}{4}S^{-1}[R] & 0\\
0 & S\end{pmatrix} \left[ \begin{pmatrix} 1 & 0\\
\frac{1}{2}S^{-1}R & 1\end{pmatrix} \right]$$

First we remark, that for a fixed positive integer $d$ there are
only finitely many $J$-equivalence classes of pairs $(T,R)$ 
with $\det(\left(\begin{smallmatrix} T & \frac{R}{2}\\ 
\frac{{}^tR}{2} & S\end{smallmatrix}\right))=d$;
therefore, the condition of the theorem implies that there exists a finite set ${\mathcal D}$ of determinants $d$
such that 
$c(T,R)\equiv 0\bmod {\mathfrak p}^m$ holds for all $(T,R)$ with
$\det (\left( \begin{smallmatrix} T & \frac{\mu}{2}\\
\frac{{}^t \mu }{2} & S\end{smallmatrix}\right))\notin {\mathcal D}.$ 
In view of $\det(T-\frac{1}{4}S^{-1}[\mu])
\det(S)=\det (\left(\begin{smallmatrix} T & \frac{\mu}{2}\\
\frac{{}^t\mu}{2} & S\end{smallmatrix} \right))$, this means that for all but
finitely many equivalence classes of the indices $T-\frac{1}{4}S^{-1}[\mu]$
the Fourier coefficients of $h_{\mu}$ are congruent zero modulo 
${\mathfrak p}^m$. Then, by Theorem \ref{ThmM2}, the $h_{\mu}$ is a constant modulo 
${\mathfrak p}^m$. On the other hand, $h_{\mu}$ can not have a nonzero constant
Fourier coefficient unless $\frac{1}{4}S^{-1}[\mu]$ is half-integral.
\qed

\section{Proof of Theorem \ref{ThmM3}}
We shall prove only the case of integral weight. 
The proof in this section works in the same way for the 
case of half-integral weight.  

We apply the following general congruence due to Rasmussen \cite{Ra} (Theorem 2.16) of Serre \cite{Se} and Katz \cite{Kat} (see also \cite{Bo-Na2}):\\
\textit{ Let $\frak{p}$ be a prime ideal with $\frak{p} | p$ for an odd prime $p$. If $f_i\in M_{k_i}(\Gamma _1^{(1)}(N))_{{\mathcal O}_\frak{p}}$ with $f_1 \equiv f_2$ mod $\frak{p}^m$ ($f_i\not \equiv 0 \bmod{\frak{p}}$) and $N$ is coprime to $p$, then $k_1\equiv k_2$ mod $(p-1)p^{\beta(m)}$.} \\

Variant (*): \textit{Let $\frak{p}$ be a prime ideal with 
$\frak{p}|p$ for an odd prime $p$.  
If $f_i\in M_{k_i}(\Gamma _0^{(1)}(Np^t),\psi )_{{\mathcal O}_{\frak{p}}}$ 
with $f_1 \equiv f_2 \bmod{\frak{p}^m}$ 
($f_i\not \equiv 0 \bmod{\frak{p}}$) and 
$N$ is coprime to $p$ as long as the conductor of 
$\psi ^2$ is coprime to $p$, then $k_1\equiv k_2$ mod 
$ \frac{p-1}{2^{\delta}}p^{\beta(m)}$, where 
$\delta$ is zero (one respectively),
if the conductor of $\psi$ is coprime to $p$ 
($p$ divides the conductor respectively). }\\  

The reason is that (by Serre \cite{Se, Se2} or in the Siegel modular forms 
setting by \cite{Bo-Na2}) any such 
$f_i$ is congruent modulo ${\mathfrak p}^m$ to a modular form $g_i$ of weight 
$k_i'$ for $\Gamma _0^{(1)}(N)$ and Nebentypus $\psi '$ with conductor of $\psi '$ 
coprime to $p$; at the same time the weights $k_i'$ and $k_i$ are congruent
modulo $\frac{p-1}{2^{\delta}}p^{\beta (m)}$.

Assume that $v _\frak{p}^{(n'+1)}(f)=v_\frak{p}^{(n')}(f)+m$. 
Taking a suitable constant multiple of $f$, we may assume 
that $v_\frak{p}^{(n')}(f)=0$, i.e. $m=v _\frak{p}^{(n'+1)}(f)
>v_\frak{p}^{(n')}(f)=0$. Now we consider 
$g:=\Phi ^{n-n'-1}(f)\in M_k(\Gamma _1^{(n'+1)}(N),\psi )_{K}$. 
This means that there exists a matrix $S_0\in \Lambda _{n'+1}$ with 
${\rm rank}(S_0)=n'$ such that $a_g(S_0)\not \equiv 0$ mod 
$\frak{p}$ and $a_g(S)\equiv 0$ mod $\frak{p}^m$ for all 
$S\in \Lambda _{n'+1}^+$. We may assume that $S_0$ is of the form
$S_0= \left(\begin{smallmatrix} 0 & 0\\
0 & T\end{smallmatrix}\right)$ for some $T\in \Lambda^+_{n'}$.
Then the Fourier-Jacobi coefficient of $g$ at the index $T$
\[\varphi _{S_0}(\tau, \frak{z})=\sum _\mu h_\mu 
(\tau)\Theta_{T} [\mu ](\tau, \frak{z})\]
satisfies 
\[h_0\equiv c\bmod{\frak{p}^m}\]
for some $c\in {\mathcal O}_{\frak{p}}^{\times}$. This $h_0$ is a 
modular form of weight $k-\frac{n'}{2}$ and level $L:=\text{level}(2T)$. 
We have to analyze $h_0$ more carefully: Its properties come from 
those of the theta series $\Theta _T[0](\tau ,\frak{z})$ and from $f$. 
Now we can show that 
\begin{Lem}
\label{Lem7}
We have $h_0^2\in M_{2k-n'}(\Gamma _0^{(1)}(NL),\psi ^2 \chi^{n'}_{-4})$, 
where $\chi _{-4}=\left(\frac{-4}{*}\right)$.   
\end{Lem}   
If we can prove this, then we can apply Variant (*). Hence we get 
$2k-n'\equiv 0$ mod $(p-1)p^{\beta(m)}$. 
Therefore, it suffices to prove Lemma \ref{Lem7}. 
\begin{proof}[Proof of Lemma \ref{Lem7}]
We write 
\[ \varphi_{T}(\tau,\frak{z})e^{2\pi i{\rm tr}(T\cdot \tau ')}={}^t{\boldsymbol h} \cdot {\boldsymbol \Theta}, \]
where ${\boldsymbol h}$ and ${\boldsymbol \Theta}$ are column vectors 
indexed by $\mu $ and with entries
${}^t {\boldsymbol h}=(h_0,\dots h_{\mu},\dots ) $ 
and
$${}^t {\boldsymbol \Theta}(Z)= (\dots , \theta_T[\mu]
(\tau,{\mathfrak z})e^{2\pi i tr(T\cdot \tau')}, \dots).
\qquad (Z\in {\mathbb H}_{n'+1})$$ 
 The general transformation theory of theta series 
tells us that for all $\gamma =\left(\begin{smallmatrix} a & b \\
c & d\end{smallmatrix}\right)
\in SL_2(\mathbb{Z})$ there exists unitary matrix $U(\gamma )$ such that 
\[{\boldsymbol \Theta }(\gamma^{\uparrow} \cdot Z)=(c\tau +d)^{\frac{n'}{2}}
U(\gamma )\cdot {\boldsymbol \Theta}. \] 
Here 
$\uparrow$ denotes the embedding of $SL_2({\mathbb Z})$ into $Sp_{n'+1}({\mathbb Z})$, given by
$$\begin{pmatrix} a & b\\ c & d\end{pmatrix}^{\uparrow}:= \begin{pmatrix}
a & 0 & b & 0\\
0 & 1_{r'} & 0 & 0{r'}\\
c & 0 & d & 0\\
0 & 0_{r'} & 0 & 1_{r'}\end{pmatrix}.$$

This implies for all $\gamma =\left(\begin{smallmatrix} a & b \\
c & d\end{smallmatrix}\right)\in \Gamma _0^{(1)}(N)$
\[ {}^t {\boldsymbol h}(\gamma\cdot\tau)\cdot U(\gamma )=\psi (d)(c\tau +d)^{k-\frac{n'}{2}}
\cdot {}^t{\boldsymbol h}(\tau). \] 

On the other hand, let $L$ be the level of $2T$ and assume that $\gamma \in \Gamma _0^{(1)}(L)$. Then the matrix $U(\gamma )$ has the form 
\[U(\gamma )=\begin{pmatrix} * & 0 & \cdots & 0 \\ * & * & \cdots & * \\ \vdots & \vdots &  & \vdots \\ * & * & \cdots & * \end{pmatrix} \]
because the ``first component'' (indexed by $\mu =0$) is itself a Jacobi 
form of weight $\frac{n'}{2}$ for $\Gamma _0^{(1)}(L)$. The matrix $U(\gamma )$ being unitary, This implies that 
\[U(\gamma )=\begin{pmatrix} * & 0 & \cdots & 0 \\ 0 & * & \cdots & * \\ \vdots & \vdots &  & \vdots \\ 0 & * & \cdots & * \end{pmatrix}, \]
hence $h_0$ is a modular form for the group $\Gamma _0^{(1)}(N)
\cap \Gamma _0^{(1)}(L)$.    
\end{proof}

\begin{Rem} Even in the case, where $n'$ is even, we cannot in general
expect the sharper congruence $k-\frac{n'}{2}\equiv 0 \bmod (p-1)p^{\beta(m)}$.
The reason is that we cannot assure that the level of $T$ is coprime to $p$;
it may bring in a quadratic character modulo $p$ and then we only have the 
weaker congruence of the theorem.
\end{Rem} 

\section{Examples}
\label{Ex}
\subsection{Trivial examples}
\label{tri}
Of course singular modular forms in the usual sense are then 
also mod $p$ singular modular forms. The rank $r$ of such a singular 
modular form
does not necessarily coincide with the $p$-rank $r_p$:
For example, let $F^n$ be the (unique) modular form 
of degree $n$, level one, weight $4$
with constant Fourier coefficient equal to $1$. 
This modular form exists in all degrees
and can be constructed as theta series for the $E_8$-lattice or 
as Eisenstein series
(possibly after analytic continuation). 
Its Fourier coefficients are all in ${\mathbb Z}$.
Now take $n$ large enough ($n>8$), then
$F^n$ is a singular modular form of rank $r=8$. 
It is then also mod $p$ singular.
Theorem \ref{ThmM3} implies that $2k-r_p=8-r_p$ is divisible by $p-1$, hence
$r_p=8$ for all $p\geq 11$, $r_7\in 
\{2,8\}$, $r_5\in \{0, 4, 8\}$ and $r_3\in\{0,2,4,6,8\}$.
More precisely, one can show that $r_5=0$ (for any prime $p\equiv 1\bmod 4$
there exists in any degree a modular form congruent $1 \bmod p$ of 
weight $p-1$,
see \cite{Bo-Na}); moreover, the lattice $E_8$ has an automorphism of
order $7$; then $r_7=2$ follows from (\ref{automorphism}).  


More generally, taking any $g\in M_{p-1}(\Gamma _n)_{\mathbb{Z}_{(p)}}$ 
with $g\equiv 1$ mod $p$ for an odd prime $p$ 
(by  \cite{Bo-Na} the existence of such forms 
is assured provided that $p>n+3$), 
we may consider $g \cdot F^n\in M_{p+3}(\Gamma _n)$. This is no longer a 
singular modular form in the usual sense. 
However, this is still a mod $p$ singular modular form 
because $g \cdot F^{n} \equiv F^n$ mod $p$. 

In this way, one can create 
mod $p$ singular modular form, which are
 nonsingular in the usual sense. Examples of this type
should be considered trivial. In the next two sections, we construct some 
examples of non-trivial mod $p$ singular modular forms in two ways.


\subsection{Construction using lattices with automorphisms}

For general properties of lattices with automorphisms we refer to
\cite{BayFlu}.\\
Let us start from an
even lattice $L$ of rank $m=2k$ with associated 
Gram matrix $S$ and assume that  
$L$ has an automorphism $\sigma$ of order $p$.

For all $T\in \Lambda^+_m$ the representation number
$$A(S,T) := \{X\in {\mathbb Z}^{(m,m)}\,\mid \, {}^t XSX=T\}$$
is divisible by $p$, because $\sigma$ acts without fixed point on this set. 
Therefore the theta series
$$\theta^m_S(Z):= \sum_{X\in {\mathbb Z}^{(m,m)}}e^{2\pi i {\rm tr}({}^tX SXZ)} \in M_k(\Gamma _m) $$ 
(modular form of weight $k$ and degree $m$) is mod $p$ singular of  
rank $r_p<m$. More precisely, let
$n_0$ be the largest $n$ such that there exists $T\in \Lambda_n^+$ with 
$A(S,T)$ not divisible by $p$.
Then $\theta^m_S$ is mod $p$ singular of $p$-rank $r_p=n_0$ and 
$m -n_0$ is divisible by $p-1$. \\

If this congruence for the representation numbers is not 
``very accidental'',
we may guess that $n_0+1$ is at the same time the smallest number $t$ such that 
$\sigma$ acts without fixed points on
$$\{X\in {\mathbb Z}^{(m,t)}\,\mid \, {\rm rank}(X)=t\}.$$

What can we say about this number $t$ from the point of view of 
lattices ?



The minimal polynomial $\mu_{\sigma}$ of $\sigma$ must be a divisor of 
$X^p-1=(X-1)\cdot \Phi_p$, where $\Phi_p$ is 
the cyclotomic polynomial $X^{p-1}+\dots +X+1$. 
In particular, the characteristic polynomial $\chi_{\sigma}$ of $\sigma$
is then of the form
$$\chi_{\sigma}= (X-1)^{\alpha} \cdot \Phi_p^{\beta}$$
with nonnegative integers $\alpha$, $\beta$. Counting degrees, this implies

\begin{equation} m=2k= \alpha+ \beta\cdot (p-1).\label{automorphism}
\end{equation}
The number $\alpha$ is then equal to  $t-1$ from above.
We cannot prove that $t-1$ and $n_0$ are the 
same, but at least both are congruent to 
$m=2k$ mod $p-1$. We do not know examples, were these
numbers do not coincide 
(i.e. where the congruence satisfied by the theta series
is really ``very accidental''.


Interesting examples of such theta series 
should come from (even unimodular) lattices $L$, 
for which the genus of $L$ contains more than one isometry class.
Otherwise, by Siegel's theorem, the properties are the same 
as those of Eisenstein series (see below).








\subsection{Construction by the Siegel-Eisenstein series}
\label{Eisen}
We denote by 
\[ \Gamma _\infty ^{(n)}:=\left\{ \begin{pmatrix}A & B \\ C & D \end{pmatrix}\in \Gamma _n \,\Big{|} \, C=0_n \right \}. \]
the Siegel parabolic subgroup.\\
For an even integer $k> n+1$, the Siegel-Eisenstein series $E_k^{(n)}\in M_k(\Gamma _n)$ is defined by the series 
\[E_k^{(n)}(Z):= \sum _{\left( \begin{smallmatrix} * & * \\ C & D \end{smallmatrix} \right) \in \Gamma _\infty ^{(n)} \setminus \Gamma _n} \det (CZ+D)^{-k}, \quad Z\in \mathbb{H}_n. \]
 
Then we have
\begin{Thm}
\label{ThmEx}
For an odd positive integer $n\ge 3$, we take a prime $p$ and an even positive integer $k>n+1$ such that 
$2k-n+1\equiv 0$ mod $p-1$ and 
\begin{align*} 
v_p \left( \frac{k}{B_k}\right)=v_p\left( \frac{k-i}{B_{2k-2i}} \right)=0 \quad for\ all\ 1\le i\le \frac{n-3}{2}. 
\end{align*} 
Then $E_k^{(n)}$ is a mod $p$ singular modular form with $p$-rank $n-1$. 
\end{Thm}

\begin{proof}
We write the Fourier expansion of $E^{(n)}_k$ as 
\[E_k^{(n)}=\sum _{0\le T\in \Lambda _n}a_k^{(n)}(T)e^{2 \pi i {\rm tr}(TZ)}. \]

Let $D^*_{m}$, $D_{m}^{**}$ be two natural numbers defined as 
\begin{align*}
D_{m}^*&:=\prod _{p|D_{m}}p^{1+v_p(k-\frac{m}{2})}\quad for\ m\equiv 0 \bmod{4},\\
D_{m}^{**}&:=\prod _{\substack{p|D_{m} \\ p\equiv3 \bmod{4}}}p^{1+v_p(k-\frac{m}{2})}\quad for\ m\equiv 2 \bmod{4},
\end{align*}
where $D_m$ is the denominator of the $m$-th Bernoulli number $B_{m}$. 

For an integer $r$ ($1\le r \le n$), put  
\begin{align*}
c_{k,r}:=
\begin{cases} 
\displaystyle 2^r \cdot \frac{k}{B_k} \cdot \prod_{i=1}^{\frac{r-1}{2}}\frac{k-i}{B_{2k-2i}}\ (r:\ odd) \\
\displaystyle 2^r \cdot \frac{k}{B_k}\cdot \frac{1}{D_{2k-r}^*}\cdot \prod_{i=1}^{\frac{r}{2}}\frac{k-i}{B_{2k-2i}}\ (r\equiv 0\bmod{4}), \\
\displaystyle 2^{r-1}\cdot \frac{k}{B_k}\cdot \frac{1}{D^{**}_{2k-r}}\cdot \prod_{i=1}^{\frac{r}{2}}\frac{k-i}{B_{2k-2i}}\ (r\equiv 2 \bmod{4}).
\end{cases}
\end{align*}
Then the results of \cite{Bo} assert that $a_{k}^{(r)}(T)\in c_{k,r} \mathbb{Z}$. 
Hence, we can write as $a_{k}^{(r)}(T)= c_{k,r} \cdot c_k^{(r)}(T)$ for 
some $c^{(r)}_k(T)\in \mathbb{Z}$. Note that $c_{k,r}\in \mathbb{Z}_{(p)}$ in this situation.    

By the choice of $p$, we have
\[v_p \left( \frac{k-\frac{n-1}{2}}{B_{2k-n+1}} \right) =1, \quad v_p \left( \frac{k}{B_k}\cdot \prod_{i=1}^{\frac{n-3}{2}}
\frac{k-i}{B_{2k-2i}}\right)=0,  \]
and hence 
\begin{align*}
a_{k}^{(n)}(T)&=2^n \cdot \left( \frac{k}{B_k}\cdot 
\prod_{i=1}^{\frac{n-3}{2}}\frac{k-i}{B_{2k-2i}}\right)
\cdot \frac{k-\frac{n-1}{2}}{B_{2k-n+1}}\cdot c^{(n)}_k(T) \equiv 0 \bmod{p}.
\end{align*}
On the other hand, the rank $n-2$ part is not zero modulo $p$: 
By our assumptions, $c_{k,n-2}$ is not divisible by $p$. Furthermore, in 
\cite[p.285]{Bo} it 
was shown that the gcd of all the $c_{k}^{(r)}(T)$ with $T\in \Lambda_r^+$ is one for all odd $r$.
We apply this to $r=n-2$ and obtain that $E^{(n)}_k$ is mod $p$ 
singular of $p$-rank $n-2$ or of $p$-rank $n-1$. The congruence of Corollary \ref{Cor3} 
shows that the $p$-rank is indeed $n-1$. This completes the proof of Theorem \ref{ThmEx}.
   
\end{proof}

\begin{Rem} Results of Weissauer \cite{Weiss} (see also Haruki \cite{Haruki} and Shimura \cite{Shi}) 
show that by 
``Hecke summation'' we can define holomorphic Eisenstein series also for small even weights $k$:
If $4\mid k$ this is true for all weights, for $k\equiv 2 \bmod 4$ it is true for
$k>\frac{n+3}{2}$ (with the same kind of Fourier expansion as in the case of large weights). 
One may then extend the result of Theorem \ref{ThmEx} above to $k\geq \frac{n}{2}$ if $4\mid k$
and for $k\equiv 2\bmod 4$ it is true for $k>\frac{n+3}{2}$.
\end{Rem}

\begin{Ex}
(Case ${n=3}$) Let $k=4$. Then $2k-n+1=6$ and hence $2k-n+1\equiv 0$ mod $p-1$ holds if $p=3$ or $7$. 
Among these primes, $p=7$ does not divide $\frac{4}{B_4}
=4\cdot (- 30)=-2^3\cdot 3\cdot 5$. 
Therefore $E^{(3)}_4$ is a mod $7$ singular modular form (see also 
subsection 6.1). 

Elsewhere, there are $(k,p)=(6,11)$, $(10,19)$ and so on. \\
(Case ${n=5}$) Let $k=6$. Then $2k-n+1=8$ and hence $2k-n+1\equiv 0$ mod $p-1$ holds if $p=3$ or $5$. In these primes, we may chose $p$ such that $p$ does not divide both $\frac{6}{B_6}=6\cdot 42=2^2\cdot 3^2 \cdot 7$ and $\frac{5}{B_{10}}=66=2\cdot 3\cdot 11$. Namely $p=5$. Therefore $E^{(5)}_6$ is a mod $5$ singular modular form.  

Elsewhere, there are $(k,p)=(8,7)$, $(8,13)$, $(10,5)$, $(10,17)$ and so on.
\end{Ex}


\subsection{Eisenstein series $E^{(n)}_{p-1}$  with  irregular prime ${\boldsymbol p}$}

We cannot expect a Siegel-Eisenstein series $E^{(n)}_{p-1}$ to be congruent 
$1 \bmod p$ for all $n$ if the prime $p$ is irregular, see \cite{Nag}.
On the other hand, the $p$-adic behavior of Bernoulli numbers seems not to 
allow for a general description of $E^{(n)}_{p-1} \bmod p$. We look instead in 
detail at the case of the smallest irregular prime, i.e. $p=37$.
Here the divisibility changes several times, which is quite remarkable.
Note that (by the theory of Hecke summation) this Eisenstein series 
$E^{(n)}_{36}$
exists for all degrees $n$ \cite{Weiss, Haruki}.
\\
The prime factorization of numerators of Bernoulli numbers 
shows that $v_{37}(B_{2j})=0$
for all $j$ with $2\leq 2j\leq 70$ with two exceptions: 
$v_{37}(B_{32})=v_{37}(B_{68})=1$.

For any odd degree $r$ with $1\leq r\leq 2p-3=71$ we look at
the numbers $c_{36,r}$ given in subsection 6.2. When viewed in 
the direction of increasing $r$,
these numbers are divisible by 37 as long as $B_{68}$ does not appear,
i.e. it is divisible by $p$ for $r=1, 3$ but no longer for $r=5$.
It becomes then divisible again by $p=37$ when $2k-2i=72-(r-1)=p-1=36$, i.e.
for $r=37$. This means that $E^{(n)}_{36}$ is $37$-singular of rank $36$ for 
all $n\geq 37$. Increasing $r$ still further, we see that $c_{36,r}$
remains divisible by $37$ until $B_{32}$ appears as $B_{2k-(r-1)}$
in the product defining $c_{36,r}$, i.e. it is no longer divisible by $p$
from $r= 41$ on until $r=71$.
This gives for $f:= E^{(n)}_{36}$ 
with $n$ any singular degree in the usual sense
(i.e. $n\geq 73$) and $p=37$: 

$$v^{(0)}_p(f)=0$$
$$v^{(1)}_p(f)=\dots = v^{(4)}_p(f)=1$$
$$v_p^{(5)}(f)=\dots = v_p^{(36)}(f)=0$$
$$v_p^{(37)}(f)=\dots = v_p^{(40)}(f)=1$$
$$v_p^{(41)}(f)= \dots = \dots v_p^{(72)}(f)=0$$
$$v_p^{(t)}(f)=\infty \,\,\mbox{for all} \,\, t\,\,\mbox{with} 
\,\,73\leq t\leq n$$

To confirm the values $v_p^{(4)}(f)$ and $v_p^{(4)}(f)$ 
one has to observe that
$D^{**}$ contains only prime factors congruent $3$ mod $4$.
Furthermore, to confirm $v_p^{(5)}(f)$ and $v_p^{(41)}(f)$ we remark that
(following \cite[p.285]{Bo}) the gcd of all $c_{k}^{(r)}(T)$ is indeed
equal to one if $r$ is odd.

The situation may be different for other irregular primes, depending
on how many among the relevant Bernoulli numbers are divisible by $p$
(and also depending on the order of divisibility).

\section{An application to Klingen-Eisenstein series}

When we rephrase our results from 
above (Theorem 3.5) using the Siegel
$\Phi$-operator, we get for a modular form $F\in M_k(\Gamma_n)$
with Fourier coefficients in a number field $K$ and a prime ideal $\frak p$ of ${\mathcal O}_K$
\begin{equation}
v^{(n)}_{\frak p}(F)\leq v_{\frak p}^{(n-j)}(\Phi^j(F))
\label{Phi-operator}
\end{equation}
provided that none of 
$$2k-(n-j),\ 2k-(n-j+1),\ \cdots ,\ 2k-(n-1)$$
is divisible by $p-1$. This applies in particular to $F:= E_{n,r}(f)$ for a cusp form 
$f=\sum_{t\in \Lambda_r^+} c(f,t)e^{2\pi i {\rm tr}(tz)}\in M_k(\Gamma_r)$, 
where $E_{n,r}(f)$ is the degree $n$ Klingen-Eisenstein series attached to $f$,
defined by
$$E_{n,r}(f)(Z):=\sum_{
M=\left(\begin{smallmatrix} A & B\\
C & D\end{smallmatrix}\right)\in C_{n,r}({\mathbb Z})\backslash \Gamma_n} 
\det(CZ+D)^{-k} f((MZ)^*) \qquad (Z\in {\mathbb H}_n);$$
here $Z^*$ denotes the $r\times r$-submatrix in the upper left corner of 
$Z$ and $C_{n,r}(\mathbb{Z})$ is an appropriate (Klingen) parabolic subgroup of $Sp_n(\mathbb{Z})$.

One can use the explicit formulas for the Fourier coefficients of
such modular forms from \cite{Bo-2,Bo-1,Mizu1} to deduce a nice 
non-divisibility result from
(\ref{Phi-operator}) about special values of certain $L$-functions: 
To get a smooth formulation, we use
the primitive Fourier coefficients instead of the ordinary Fourier coefficients,
see \cite{Bo-Ragh}: Then the primitive Fourier coefficients of 
$E_{n,r}(f)$ are given (up
to some factor not depending on $T$ and up to a critical value of the standard $L$-function $D(f,s)$ by a special value of the Rankin convolution 
${\mathcal R}(f,\theta^r _T)$ of $f$ and
the degree $r$ theta series attached to $T$, defined by
$${\mathcal R}(f,\theta^r_T,s):=\sum_{t\in GL_n({\mathbb Z})\backslash
\Lambda_r^+} \frac{c(f,t) A(T,t)}
{\epsilon(t)\det(t)^s},$$
where $\epsilon(T)$ is the number of units of $t$. 
For the precise formula, see (\ref{primitive}) below.
We just reformulate, what (\ref{Phi-operator}) says for 
Klingen-Eisenstein series, when expressed for the (primitive) 
Fourier coefficients:

\begin{Thm}
\label{Kli}
Assume $f=\sum c(f,t)e^{2\pi i {\rm tr}(tz)}\in M_k(\Gamma_r)$ is a cuspidal 
eigenform of all Hecke operators, $k>n+r+1$. 
Let $\frak p$ be a prime ideal of ${\mathcal O}_K$, where
$K$ is a number field containing the Fourier coefficients of $f$.
For any $T\in \Lambda_n^+$ we put

\begin{equation}b(f,T):= \alpha(k,r) \cdot a^{(n)}_k(T)^*\cdot 
\frac{{\mathcal R}(f,\theta^r_T,k-\frac{r+1}{2})}{D(f,k-r)}.
\label{primitive}\end{equation}

Let $A^n(f)$ the ${\mathcal O}_K$-module generated by
all the $b(f,T)$, $T\in \Lambda_n^+$.
Then

$$v_{\frak p} (A^n(f))\leq v^{(r)}_{\frak p}(f)$$
provided that none of 
$$2k-r,\ 2k-(r+1),\ \cdots ,\ 2k-(n-1)$$
is divisible by $p-1$.
\end{Thm} 
Here $a_k^{(n)}(T)^*$ denotes a primitive Fourier coefficient 
of the Siegel-Eisenstein series
$E^{(n)}_k(Z)=\sum a^{(n)}_k(T)e^{2\pi i{\rm tr}(TZ)}$ and
$D(f,s)$ is the standard $L$-function attached to $f$ (with Euler factors
of degree $2n+1$). Furthermore the numerical constant $\alpha(k,r)$ is given by
$$\alpha(k,r):= \frac{1}{2}\zeta(k)\prod_{i=1}^r\zeta(2k-2i)$$ 

Mizumoto \cite{Mizu2} obtained 
subtle integrality results for the Fourier coefficients
of Klingen-Eisenstein series, i.e. he obtained lower bounds for $v_{\mathfrak p}(A^{n}(f))$. 
The theorem above can be viewed as 
supplement to Mizumoto's result in the sense
that it describes lower bounds for $v_{\mathfrak p}(A^n(f))$.
It may be difficult to get such a result directly 
from inspecting the series above.

\begin{Rem}
Theorem \ref{Kli} applies also to the case $r=0$, $f=1$:
Then we deal just with the ordinary Siegel-Eisenstein series. 
If $n$ is odd, our result above can in this case be deduced in a more direct 
way from properties of Bernoulli numbers as follows: We shall prove 

\textit {Let $n$ be an odd positive integer. If all of \[ 2k-r,\ 2k-(r+1),\ \cdots ,\ 2k-(n-1)\]
are not divisible by $p-1$, then $v _p^{(n)}(E_k^{(n)})\le v ^{(r)}_p(E_k^{(n)})$. }

\begin{proof}
We write $B_k:=\frac{N_k}{D_k}$. \\ 
For example, assume that $r$ is even. Under the notation as in subsection \ref{Eisen}, we have  
\begin{align*}
v _p(c_{k,r})&=v _p\left( \frac{k}{B_k} \right) +\sum _{i=1}^{\frac{r}{2}} v _p\left( \frac{k-i}{B_{2k-2i}} \right)-v _p(D_{2k-r}) \\  
&=v _p\left( \frac{k}{B_k} \right) +\sum _{i=1}^{\frac{r}{2}} v _p\left( \frac{k-i}{B_{2k-2i}} \right)
\end{align*}
because of $v _p(D_{2k-r})=0$. Hence 
\begin{align*}
v_p(c_{k,n})&=v _p\left( \frac{k}{B_k} \right) +\sum _{i=1}^{\frac{n-1}{2}} v_p\left( \frac{k-i}{B_{2k-2i}} \right) \\
&=v_p(c_{k,r})+\sum _{i=\frac{r}{2}+1}^{\frac{n-1}{2}}v _p\left( \frac{k-i}{B_{2k-2i}} \right)\\
&=v _p(c_{k,r})+\sum _{i=\frac{r}{2}+1}^{\frac{n-1}{2}}(\nu _p(k-i)-v _p(N_{2k-2i})), 
\end{align*}
where the last equality follows from $v _p(D_{2k-2i})=0$ for all $\frac{r}{2}+1\le i\le \frac{n-1}{2}$. 
Since $v _p(N_{2k-2i})\ge v _p\left( k-i \right)$, we have $v _p(c_{k,n})\le v _p(c_{k,r})$. The existence $T\in \Lambda _n^+$ such that $c_k^{(n)}(T)\not \equiv 0$ mod $p$ implies $v _p^{(n)}(E_k^{(n)})\le v ^{(r)}_p(E_k^{(n)})$. Similarly, we can prove the case where $r$ is odd. 
\end{proof}
\end{Rem}

\section*{A final remark}
Most of our results were given only for non-dyadic prime ideals. However, using a more complicated formulation, one can give versions of our results including dyadic primes. A new feature is the occurrence of the quadratic character modulo $4$ in Lemma \ref{Lem7}, which further weakens the result for
 dyadic primes. We omit details.

\begin{flushleft}
Siegfried B\"ocherer\\ 
Graduate School of Mathematical Sciences \\
The University of Tokyo\\
Komaba, Tokyo 153-8914, Japan\\
Email: boecherer@t-online.de \vspace{3mm}

Toshiyuki Kikuta\\ 
College of Science and Engineering \\
Ritsumeikan University\\
1-1-1 Noji-higashi, Kusatsu, Shiga, 525-8577, Japan\\
Email: kikuta84@gmail.com
\\
\end{flushleft}

\end{document}